\documentclass[12pt,reqno]{amsart}
\usepackage[english]{babel}
\usepackage[cp1251]{inputenc}

\usepackage{mathrsfs}
\usepackage{amsthm}
\usepackage{amssymb}
\usepackage{amsmath,amsthm,amsfonts}
\usepackage{mathabx}
\usepackage{hhline}
\usepackage{textcomp}
\usepackage{amsmath,amscd}
\usepackage[all,cmtip]{xy}
\usepackage{mathtools}
\usepackage[margin=1in]{geometry}

\author[]{Kaveh Eftekharinasab}
\title[]{The Morse-Sard-Brown Theorem for Functionals on Bounded-Fr\'{e}chet-Finsler Manifolds}
\address{Topology dept. \\ Institute of Mathematics of NAS of Ukraine \\ Te\-re\-shchen\-kivska st. 3, Kyiv, 01601 Ukraine}
\email{kaveh@imath.kiev.ua}

\makeatletter
\@addtoreset{equation}{section}
\@addtoreset{figure}{section}
\@addtoreset{table}{section}
\makeatother

\newtheorem{theorem}[subsection]{Theorem}
\newtheorem{lemma}[subsection]{Lemma}
\newtheorem{remk}[subsection]{Remark}
\newtheorem{prop}[subsection]{Proposition}
\newtheorem{defn}[subsection]{Definition}

\newtheorem*{con}{Condition}

\theoremstyle{definition}

\DeclareMathAlphabet{\mathpzc}{OT1}{pzc}{m}{it}

\newcommand{\rr}{\mathbb{R}}
\newcommand{\nn}{\mathbb{N}}
\newcommand\Finsler{\parallel \cdot \parallel^n}
\DeclareMathOperator{\codim}{codim}
\DeclareMathOperator{\Ind}{Ind}
\DeclareMathOperator{\Img}{Img}

\DeclareMathOperator{\Aut}{Aut}
\DeclareMathOperator{\Iso}{Iso}

\begin{document}

\begin{abstract}
In this paper we study Lipschitz-Fredholm vector fields on Bounded-Fr\'{e}chet-Finsler manifolds. 
In this context we generalize the Morse-Sard-Brown theorem, asserting that if $M$ is a  connected smooth bounded-Fr\'{e}chet-Finsler manifold endowed with 
a strengthened connection $\mathcal{K}$ and if $\xi$ is a smooth Lipschitz-Fredholm vector field on $M$ with respect to $\mathcal{K}$ which satisfies condition (CV).
Then, for any smooth functional $l$ on $M$ which is associated to $\xi$, the set of the critical values of $l$
is of first category in $\rr$. Therefore, the set of the regular values of $l$ is a residual Baire subset of $\mathbb{R}$.
\end{abstract}
\maketitle

\section{Introduction}
The notion of a Fredholm vector field on a Banach manifold $ B$ with respect to a connection on $B$ was introduced by Tromba~\cite{tromba}.
Such vector fields arise naturally in non-linear analysis from variational problems. There are geometrical objects such as harmonic maps, geodesics
and minimal surfaces which arise as the zeros of a Fredholm vector field. Therefore, it would be valuable to study the critical
points of functionals  which are associated to Fredholm vector fields.
In~\cite{tromba3}, Tromba proved the Morse-Sard-Brown theorem for this type of functionals in the case of Banach manifolds.
Such a theorem would have applications to problems in the calculus of variations in the large such as
Morse theory~\cite{tromba2} and  index theory~\cite{tromba}.

The purpose of this paper is to extend the theorem of Tromba (\cite[Theorem 1\textasciiacute(MSB)]{tromba3}) to a new class of generalized Fr\'{e}chet manifolds, 
the so-called bounded Fr\'{e}chet manifolds, which was introduced in~\cite{m}. Such spaces arise in geometry and physical field theory and have
many desirable properties. For instance, the space of all smooth sections of a fiber bundle (over closed or non-compact manifolds), which
is the foremost example of infinite dimensional manifolds, has the structure of a bounded Fr\'{e}chet manifold, see~\cite[Theorem 3.34]{m}.  
The idea to introduce this category of manifolds was to overcome some permanent difficulties (i.e. problems of intrinsic nature) in the theory of Fr\'{e}chet spaces. 
For example, the lack of a non-trivial topological group structure on the general linear group of a Fr\'{e}chet space. As for the importance of bounded Fr\'{e}chet manifolds, we refer
to~\cite{k},~\cite{ka} and~\cite{m}. 

Essentially, to define the index of Fredholm vector fields we need the stability of Fredholm operators under small perturbation but 
 this is unobtainable in the case of proper Fr\'{e}chet spaces (non-normable spaces) in general, see~\cite{k}.
Also, we need a subtle notion of a connection via a connection map but (because of the aforementioned problem) such a 
connection can not be constructed for Fr\'{e}chet manifolds in general (cf.~\cite{dod}). 
However, in the case of bounded Fr\'{e}chet manifolds under the global Lipschitz assumption on Fredholm operators,
the stability of Lipschitz-Fredholm operators was established in~\cite{k}. In addition, the notion of a connection via a connection map
was defined in~\cite{ka}. 
By using these results, we introduce the notion of a Lipschitz-Fredholm vector field in Section~\ref{fvf}. With regard to a kind of 
compactness assumption (condition (WCV)), which one needs to impose on vector fields, we will be interested in manifolds which admit a Finsler structure.
We then define Finsler structures for bounded Fr\'{e}chet manifolds in Section~\ref{fin}.
Finally, after we explained all subsequent portions for proving the Morse-Sard-Brown theorem,
we formulate the theorem in the setting of Finsler manifolds in Section~\ref{msb}. A key point in the proof of the theorem is 
Proposition~\ref{representation} which in its simplest form says that a Lipschitz-Fredholm
vector field $\xi $ near $origin$ locally has a representation of the form $\xi(u,v) = (u,\eta(u,v))$, where $\eta$ is some smooth map. Indeed, this is
a consequence of the inverse function theorem (Theorem~\ref{invr}). One of the most important advantage of the category of bounded Fr\'{e}chet manifold
is the availability of the inverse function theorem in the sense of Nash and Moser (see~\cite{m}).

Morse theory and index theories for Fr\'{e}chet manifolds have not been developed. Nevertheless, our approach provides some  essential tools 
(such as connection maps, covariant derivatives, Finsler structures) which would create a proper framework for these theories.
\section{Preliminaries}
 In this section we summarize all the necessary preliminary material that we need
for a self contained presentation of the paper. We shall work in the category of smooth manifolds and bundles. We refer
to~\cite{ka} for the basic geometry of bounded Fr\'{e}chet manifolds. 

A Fr\'{e}chet space $(F,d)$ is a complete metrizable locally convex space whose 
topology is defined by a complete translational-invariant metric $d$. A metric with absolutely convex balls will be called a standard metric. 
Note that every Fr\'{e}chet space  admits a standard metric which defines its topology: If $\alpha_n$ is an arbitrary sequence of 
positive real numbers converging to $zero$ and if $\rho_n$ is any sequence of continuous seminorms defining the 
 topology of $F$. Then
\begin{equation}
d_{\alpha,\, \rho}(e,f)\coloneq \sup_{n \in \mathbb{N}} \alpha _n \dfrac{\rho_n(e-f)}{1+\rho_n(e-f)}
\end{equation}
is a metric on $F$ with the desired properties. We shall always define the topology of Fr\'{e}chet spaces with this type of metrics.
Let $(E,g)$ be another Fr\'{e}chet space and let $\mathcal{L}_{g,d}(E,F)$ be the set of all linear maps $ L: E \rightarrow F $ such that 
\begin{equation*}
\mathpzc{Lip} (L )_{g,d}\, \coloneq \displaystyle \sup_{x \in E\setminus\{0\}} \dfrac{d (L(x),0)}{g( x,0)} < \infty.
\end{equation*}
We abbreviate $\mathcal{L}_{g}(E)\coloneq \mathcal{L}_{g,g}(E,E)$ and write $\mathpzc{Lip}(L)_{g}\, = \,\mathpzc{Lip}(L)_{g,g}$ for 
$L \in \mathcal{L}_{g}(E) $.
 The metric $D_{g,d}$ defined by
\begin{equation} \label{metric}  
 D_{g,d}: \mathcal{L}_{g,d}(E,F) \times \mathcal{L}_{g,d}(E,F) \longrightarrow [0,\infty) , \,\,
(L,H) \mapsto \mathpzc{Lip}(L -H)_{g,d} \,,
\end{equation}
is a translational-invariant metric on $\mathcal{L}_{d,g}(E,F)$ turning it into an Abelian topological group (see~\cite[Remark 1.9]{glockner}). The latter is not a topological vector space
in general, but a locally convex vector group with absolutely convex balls. The topology on $\mathcal{L}_{d,g}(E,F)$  will always be defined by the metric $D_{g,d}$.
We shall always equip the product of any finite number $k$ of Fr\'{e}chet spaces $(F_i,d_i), 1 \leq i \leq k$, 
with the maximum metric 
$$d_{max}((x_1,\dots,x_k),(y_1,\dots,y_k)) \coloneq \max _{1 \leq i \leq k} d_i(x_i,y_i).$$
Let $ E,F $ be Fr\'{e}chet spaces, $ U $ an open subset of $ E $, and $ P:U \rightarrow F $
 a continuous map. Let $CL(E,F)$ be the space of all continuous linear maps from $E$ to $F$ topologized by the compact-open topology. 
 We say $ P $ is differentiable at the point $ p \in U$ if there exists a linear map
$\operatorname{d}P(p): E \rightarrow F$ with $$\operatorname{d}P(p)h =  \lim_{t\rightarrow 0}\dfrac{P(p+th)-P(p)}{t},$$ for all $ h \in E $. 
If $P$ is differentiable at all points $p \in U$, if $\operatorname{d}P(p) : U \rightarrow CL(E,F)$ is continuous for all
$p \in U$ and if the induced map $ P': U \times E \rightarrow F,\,(u,h) \mapsto \operatorname{d}P(u)h $
 is continuous in the product topology, then we say that $ P $ is Keller-differentiable.
 We define $ P^{(k+1)}: U \times E^{k+1} \rightarrow F $ inductively by
\begin{equation*}
P^{(k+1)}(u,f_{1},...,f_{k+1}) = \lim_{t\rightarrow 0}\dfrac{P^{(k)}(u+tf_{k+1})(f_1,...,f_k)- P^{(k)}(u)(f_1,...,f_k)}{t}.
\end{equation*}

If $P$ is Keller-differentiable, $ \operatorname{d}P(p) \in \mathcal{L}_{d,g}(E,F) $ for all $ p \in U $, and the induced map 
$ \operatorname{d}P(p) : U \rightarrow \mathcal{L}_{d,g}(E,F)   $ is continuous, then $ P $ is called b-differentiable. We say $ P $ is $ MC^{0} $ and write $ P^0 = P $ if it is continuous. 
We say $P$ is an $ MC^{1} $ and write  $P^{(1)} = P' $ if it is b-differentiable. Let $ \mathcal{L}_{d,g}(E,F)_0 $ be 
the connected component of $ \mathcal{L}_{d,g}(E,F) $ containing the zero map. If $ P $ is b-differentiable and if 
 $V \subseteq U$ is a connected open neighborhood of $x_0 \in U$, then $P'(V)$ is connected and hence contained in the connected component
$P'(x_0) +  \mathcal{L}_{d,g}(E,F)_0 $ of $P'(x_0)$ in $\mathcal{L}_{d,g}(E,F)$. Thus, $P'\mid_V - P'(x_0):V \rightarrow \mathcal{L}_{d,g}(E,F)_0 $
 is again a map between subsets of Fr\'{e}chet spaces. This enables a recursive definition: If $P$ is $MC^1$ and $V$ can be chosen for each
 $x_0 \in U$ such that $P'\mid_V - P'(x_0):V \rightarrow \mathcal{L}_{d,g}(E,F)_0 $ is $ MC^{k-1} $, 
then $ P $ is called an $ MC^k$-map. We make a piecewise definition of $P^{(k)}$ by $ P^{(k)}\mid_V \coloneq \left(P'\mid_V - P'(x_0)\right)^{(k-1)} $
for $x_0$ and $V$ as before. 
The map $ P $ is $ MC^{\infty} $ if it is $ MC^k $ for all $ k \in \mathbb{N}_0 $. We shall denote the derivative of $P$ at $p$ by $\operatorname{D}P(p)$.

A bounded Fr\'{e}chet manifold is a Hausdorff second countable topological space with an atlas of coordinate 
charts taking their values in Fr\'{e}chet spaces such that the coordinate transition functions are all
 $ MC^{\infty} $-maps. 

\section{Lipschitz-Fredholm vector fields}\label{fvf}

Throughout the paper we assume that $(F,d)$ is a Fr\'{e}chet space and $M$ is a bounded Fr\'{e}chet manifold modelled on $F$.
Let $\{ (U_\alpha,\varphi_\alpha)\}_{\alpha \in \mathcal{A}}$ be a compatible atlas for $M$. The latter gives rise to a trivializing atlas $\{ (\pi_M^{-1}(U_\alpha),\psi_\alpha)\}_{\alpha \in \mathcal{A}}$ on
the tangent bundle $ \pi_M : TM \rightarrow M$, with
\begin{equation*}
 \psi_\alpha : \pi_M^{-1}(U_\alpha) \rightarrow \varphi_\alpha (U_\alpha)\times F , \quad j^1_p(\mathpzc{f}) \mapsto (\varphi_\alpha(p), (\varphi_\alpha \circ \mathpzc{f})'(0)),
\end{equation*}
where $j^1_p(\mathpzc{f})$ stands for the $1$-jet of an $MC^{\infty}$-mapping $\mathpzc{f}:\mathbb{R}\rightarrow M$ that sends $zero$ to $ p \in M$. 
Let $N$ be another bounded Fr\'{e}chet manifold and  $h: M \rightarrow N$  an $MC^k$-map. 
 The tangent map $Th :TM \rightarrow TN$ is defined by $T h(j^1_p(\mathpzc{f}))  = j^1_{h(p)}(h\circ \mathpzc{f})$.
Let $\Pi_{TM} : T(TM) \rightarrow TM$ be an ordinary tangent bundle  over $TM$ with 
the corresponding trivializing atlas $\{ ( \Pi^{-1}_{TM} (\pi_M^{-1}(U_\alpha)),\widetilde{\psi}_\alpha)\}_{\alpha \in \mathcal{A}}$. 
A strengthened connection map on the tangent bundle $TM$ (possible also for general vector bundles) was defined in~\cite{ka}.
A strengthened connection map for $ T{M} $ is a map $ \mathcal{K}: T(T{M}) \to T{M}  $, which is fully
determined by its local form:
\begin{gather*}
	\mathcal{K}_{\alpha}\coloneq \varphi_\alpha \circ \mathcal{K} \circ (\widetilde{\varPhi}_\alpha))^{-1},\\
	\varphi_\alpha(U_\alpha) \times {F} \times {F}\times  {F} \to \varphi_\alpha(U_\alpha) \times {F},\quad
	\mathbb{K}_{\alpha} = (f,g,h,k)=  (f,k+\tau_{\alpha}(f,g) h),
\end{gather*}
for a family of  mappings $$ \tau_{\alpha}: \varphi_\alpha(U_\alpha) \times {F}  \to \mathcal{L}_{{d}}({F})^{\times} .$$
Here $ \mathcal{L}_{{d}}({F})^{\times} $ is a subset of $ \mathcal{L}_{{d}}({F}) $ consists of invertiable mappings.
The mapping $ \tau_{\alpha} $ is $ MC^{k-1} $ in the sense that the map
\begin{equation*}
	\widehat{\tau_{\alpha}}: (\varphi_\alpha(U_\alpha) \times {F}) \times {F} \to {F} \times {F},\quad
	(x,y,h) \mapsto (\tau_{\alpha}(x,y)(h), \tau_{\alpha}^{-1}(x,y)(h))
\end{equation*}
is $ MC^{k-1} $. It follows of course that $ \mathbb{K} $ is of class $ MC^{k-1} $. 
 A strengthened connection on $M$ is a strengthened connection
map on the tangent bundle $\pi_M : TM \rightarrow M$. 
A strengthened connection $\mathcal{K}$ is linear if and only if it is linear on the fibers of the tangent map. Locally $T\pi$ is the map $U_{\alpha}\times F \times F \times F \rightarrow U_{\alpha} \times F$ 
defined by $T\pi(f,\xi,h,\gamma)= (f,h)$, hence locally its fibers are the spaces $\{f\}\times F \times \{h\} \times F$. Therefore, $\mathcal{K}$ is linear on these fibers
if and only if the maps $(g,k)\mapsto k+ \tau_{\alpha}(f,g)h$ are linear, and this means that the mappings $\tau_{\alpha}$ need to be linear
with respect to the second variable. 
\begin{remk}
 If $\varphi : U \subset M \rightarrow F$ is a local coordinate chart for $M$, then a vector field $\xi$ on $M$ induces a vector field
 $\xi$ on $F$ called the local representative of $\xi$ by the formula $\xi(x)=T\varphi\cdotp \xi(\varphi^{-1}(x))$. We shall frequently use $\xi$ itself
 to denote this local representation.
\end{remk}
In the following we adapt the Elliason's definition of covariant derivative~\cite{eliasson}.
\begin{defn}\label{covariant} 
 Let $\pi_M : TM \rightarrow M$ be the tangent bundle over $M$. Let $N$ be a bounded Fr\'{e}chet manifold modelled on $F$, 
 $\lambda: N \rightarrow M$  a Fr\'{e}chet vector bundle with fiber  $F$, and $K_{\lambda}$ a strengthened connection map on $TN$. If $\xi : M \rightarrow N$ is a smooth section of $\lambda$, 
 we define the
 covariant derivative of $\xi$ at $p \in M$ to be the bundle map $\nabla \xi: TM \rightarrow N$ given by
 \begin{equation*}
  \nabla \xi (p) = K_{\lambda} \circ T_p\xi,\, T_p\xi = T\xi|_{T_pM}.
 \end{equation*}
In a local coordinate chart $(U,\varPhi)$ it becomes
\begin{equation*}
 \nabla \xi (x) \cdotp y = \operatorname{D}\xi (x) \cdotp y + \tau_{\varPhi}(x, \xi (x))y.
\end{equation*}
Where $\tau_{\varPhi}$ is the component for $K_{\lambda}$ with respect to the chart $(U,\varPhi)$.
\end{defn}
The covariant derivative $\nabla \xi (p)$ is a linear map from the tangent space $T_pM$ to $F_p \coloneq \lambda^{-1}(p)$. This is because it is the combination of
the tangent map $T_p \xi$ that maps $T_pM$ linearly into $T_{\xi(p)} N$ with $K_{\lambda}$ which is a linear map from $T_{\xi(p)}N$
to $F_p$. 

\begin{defn}[\cite{k}, Definition 3.2]
Let $ (F,d) $ and $ (E,g) $ be Fr\'{e}chet spaces. 
A map $ \varphi \in \mathcal{L}_{g,d}(E,F) $ is called Lipschitz-Fredholm operator if it satisfies 
the following conditions:
\begin{enumerate}
\item The image of $ \varphi $ is closed.
\item The dimension of the kernel of $ \varphi$ is finite.
\item The co-dimension of the image of $ \varphi $ is finite.
\end{enumerate}
\end{defn}
We denote by $ \mathcal{LF}(E,F) $ the set of all Lipschitz-Fredholm operators from $ E $ into
$ F $. For $ \varphi \in \mathcal{LF}(E,F) $ we define the index of $ \varphi $ as follows:
\begin{equation*}
\Ind \varphi = \dim \ker \varphi - \codim \Img \varphi.
\end{equation*}
\begin{theorem}[\cite{k}, Theorem 3.2]\label{th}
$ \mathcal{LF}(E,F) $ is open in $\mathcal{L}_{g,d}(E,F)$ with respect to the topology defined by the Metric
 \eqref{metric}. Furthermore, the function $ T \rightarrow \Ind T $ is continuous on $ \mathcal{LF}(E,F) $,
 hence constant on connected components of $ \mathcal{LF}(E,F) $.
\end{theorem}
Now we define a Lipschitz-Fredholm vector field on $M$ with respect to a connection on $M$. 
\begin{defn}\label{lipfred}
 A smooth vector field  $\xi : M \rightarrow TM$ is called Lipschitz-Fredholm with respect to a strengthened connection
 $\mathcal{K} :T(TM) \rightarrow TM$  if for each $p \in M$, $\nabla \xi(p): T_pM \rightarrow T_pM$ is a linear Lipschitz-Fredholm operator. The 
 index of $\xi$ at $p$ is defined to be the index of $ \nabla \xi (p)$, that is
 \begin{equation*}
  \Ind \nabla \xi(p) = \dim \ker \nabla \xi(p) - \codim \Img \nabla \xi(p).
 \end{equation*}
 By Theorem~\ref{th}, if $M$ is connected the index is independent of the choice of $p$ and the common integer is called the index of $\xi$, while
 if $M$ is not connected the index is constant on components and we shall require it to
be the same on all components of $M$.
\end{defn}

\begin{remk} \label{zero}
Note that the notion of Lipschitz-Fredholm vector field depends on the choice of  $\mathcal{K}$.
 If $p$ is  a $zero$ of $\xi$, $\xi(p) =0$, then by Definition~\ref{covariant} we have $\nabla \xi (p)  = \operatorname{D}\xi (p)$
 and hence the covariant derivative at $p$ does not depend on $\mathcal{K}$. In this case, the derivative of 
 $\xi$ at $p$, $ \operatorname{D}\xi (p)$, can be viewed as a linear endomorphism from $T_pM$ into itself. 
\end{remk}

\section{ Finsler structures}\label{fin}
 A Finsler structure on the bounded Fr\'{e}chet manifold $M$ is defined in the same way as in the
 case of Fr\'{e}chet manifolds (see~\cite{bejan} for the definition of Fr\'{e}chet-Finsler manifolds). 
 However, we need a countable family of seminorms on its Fr\'{e}chet model space $F$ which defines the topology of $F$. As mentioned
in Preliminaries, we always define the topology of a Fr\'{e}chet space by a metric with absolutely convex balls. One reason for this consideration is 
that a metric with this property can give us back original seminorms. More precisely:
\begin{remk}[\cite{m}, Theorem 3.4] \label{semi}
 Assume that $(E,g)$ is a Fr\'{e}chet space and $g$ is a metric with absolutely convex balls. Let $B_{\frac{1}{i}}^g(0) \coloneq \{ y \in E  \mid g(y,0) < \tfrac{1}{i}\}$, and suppose $U_i$'s, $i \in \nn$, are convex 
 subsets of $B_{\frac{1}{i}}^g(0)$. Define the Minkowski functionals
 $$
  \parallel v \parallel_i \coloneq \inf \{ \epsilon > 0 \mid \epsilon \in \mathbb{R},\,\dfrac{1}{\epsilon} \cdot v \in U_i\}. 
 $$
These Minkowski functionals are continuous seminorms on $E$. A collection $\{\parallel v \parallel_i\}_{i \in \nn}$ of these seminorms gives the topology of $E$. 
\end{remk}
\begin{defn}\label{defni}
Let $F$ be as before. Let $X$ be a topological space and $V = X \times F$ the trivial bundle with fiber $F$ over $X$. A Finsler
structure for $V$ is a collection of functions $ \Finsler: V \rightarrow \mathbb{R}^+$, $n \in \nn$, such that 
\begin{enumerate}
 \item For $b \in X$ fixed, $\parallel (b,x)\parallel^n = \parallel x \parallel_b^n$ is a collection of seminorms 
 on $F$ which gives the topology of $F$.
  \item Given $K>1$ and $x_0 \in X$, there exits a neighborhood $\mathcal{U}$ of $x_0$  such that 
 \begin{equation} \label{ine}
 \dfrac{1}{K}\parallel f \parallel^n_{x_0} \,\leqq \,\parallel f \parallel^n_{x}\, \leqq K \parallel f \parallel^n_{x_0}
 \end{equation}
for all $x \in \mathcal{U}$, $n \in \nn$, $f \in F$.
\end{enumerate}
\end{defn}
Let $\pi_M : TM \rightarrow M$ be the tangent bundle and let $\Finsler: TM \rightarrow \mathbb{R}^+$ be a collection of functions, $n \in \nn$. We say 
$\{\Finsler\}_{n \in \nn} $ is a Finsler structure for 
$TM$ if for a given $m_0 \in M$ and any open neighborhood $U$ of $m_0$  which trivializes 
the tangent bundle $TM$, i.e.
$$\psi:  \pi_M^{-1}(U)  \approx U \times (F_{m_0} \coloneq \pi_M^{-1}(m_0) ), $$
$\{\Finsler \circ \, \psi^{-1}\}_{n \in \nn}$ is a Finsler structure for $U \times F_{m_0}$.
\begin{defn}
 A bounded-Fr\'{e}chet-Finsler manifold is a bounded Fr\'{e}chet manifold together with a Finsler structure on its tangent bundle.
\end{defn}

\begin{prop}
 Let $N$ be a paracompact bounded Fr\'{e}chet manifold modelled on  a Fr\'{e}chet space $(E,g)$. If 
 all seminorms $\parallel \cdot \parallel_i$, $i \in \nn$, $($which are defined as in Remark~\ref{semi}$)$ are smooth maps on $E\setminus \{ 0\}$, then
 $N$ admits a partition of unity. Moreover, $N$ admits a Finsler structure.
\end{prop}
\begin{proof}
 See~\cite{bejan}, Propositions 3 and 4.
\end{proof}
If $\{ \Finsler\}_{n \in \nn}$ is a Finsler structure for $M$ then eventually we can obtain a graded Finsler structure, denoted by $(\Finsler)_{n \in \nn}$, for $M$.
Let $(\Finsler)_{n \in \nn}$ be a graded Finsler structure for $M$.
We define the length of piecewise $MC^1$-curve $\gamma : [a,b] \rightarrow M$ by 
\begin{equation*}
 L^n(\gamma) = \int_a^b \parallel \gamma'(t) \parallel_{\gamma(t)}^n\, \mathrm{d}t.
\end{equation*}
On each connected component of $M$, the distance is defined by
\begin{equation*} 
 \rho^n (x,y) = \inf_{\gamma} L^n(\gamma),
\end{equation*}
where infimum is taken over all continuous piecewise $MC^1$-curve connecting $x$ to $y$. Thus,
we obtain an increasing sequence of pseudometrics $\rho^n(x,y)$ and define the distance $\rho$ by
\begin{equation}\label{finmetric}
\rho (x,y) = \displaystyle \sum_{n = 1}^{n = \infty} \dfrac{1}{2^n} \cdot \dfrac {\rho^n(x,y)}{ 1+ \rho^n(x,y)}.
\end{equation}
\begin{lemma}[\cite{bejan}, Lemma 2]\label{psm}
A collection $\{\sigma^i\}_{i \in \nn}$ of pseudometrics on $F$ defines a unique topology $\mathcal{T}$ such that for every
sequence $(x_n)_{n \in \nn} \subset F$, we have $x_n \rightarrow x$ in topology $\mathcal{T}$ if and only if $\sigma^i(x_n,x) \rightarrow 0$, for
all $i \in \nn$. The topology is Hausdorff if and only if $x=y$ when all $\sigma^i(x,y)=0$. In addition, 
$$
\sigma (x,y) = \displaystyle \sum_{n = 1}^{n = \infty} \dfrac{1}{2^n} \cdot \dfrac {\sigma^n(x,y)}{ 1+ \sigma^n(x,y)}
$$
is a pseudometric on $F$ defines the same topology.
\end{lemma}
With the aid of this lemma, the proof of the following theorem is close to the usual proof given for Banach manifolds (cf.~\cite{palais2}).
\begin{theorem}
 Suppose $M$ is connected and endowed with a Finsler structure $\{\Finsler\}_{n \in \nn}$. 
 Then the distance $\rho$ defined by~\eqref{finmetric} is a metric for $M$. Furthermore, the topology induced by this metric coincides with the original topology of $M$. 
\end{theorem}
\begin{proof}
 The distance $\rho$ is pseudometric by Lemma~\ref{psm}. We prove that $\rho(x_0,y_0) > 0$ if $x_0\neq y_0$. 
 Let $\{ \parallel \cdot \parallel_n\}_{n \in \nn}$ be a collection of all seminorms on $F$ (which are defined as in Remark~\ref{semi}).
 Given $x_0 \in M$, let $\varphi: U \rightarrow F$
 be a chart for $M$ with $x_0 \in U$ and $\varphi(x_0) = u_0$. Let $y_0 \in M$ and $\gamma(t)$ an $MC^1$-curve $\gamma : [a,b] \rightarrow M$ connecting $x_0$ to $y_0$. 
 Let $B_r(u_0)$ be a ball with center $u_0$ and radius $r>0$.
 Choose  $r$ small enough so that $\mathcal{U} \coloneq \varphi^{-1}(B_r(u_0)) \subset U$ and for a given $K>1$
 \begin{equation*}
 \dfrac{1}{K}\parallel f \parallel^n_{x_0} \,\leqq \,\parallel f \parallel^n_{x}\, \leqq K \parallel f \parallel^n_{x_0}
 \end{equation*}
 for all $x \in \mathcal{U}$, $n \in \nn$, $f \in F$. 
 Let $I = [a,b]$ and $ \mu (t) = \varphi \circ \gamma (t)$.
 If $\gamma(I) \subset \mathcal{U}$, then let $\beta = b$. Otherwise, 
 let $\beta$ be the first $t>0$ such that $\parallel \mu(t) - u_0  \parallel_n= r$ 
 for all $n \in \nn$. Then, since for $x \in U$ the map $\phi(x) : T_{x}M \rightarrow F$ given by
 $j_{x}^1 \mapsto \varphi (x) $ is a homeomorphism it follows that
 for all $ n \in \nn$  
 \begin{align*}
\int_a^b \parallel \gamma'(t) \parallel_{\gamma(t)}^n\, \mathrm{d}t &\geq \dfrac{1}{K} \int_a^{\beta} \parallel \phi^{-1}(x) \circ \mu'(t) \parallel_{x_0}^n \mathrm{d}t 
\geq k_1 \int_a^{\beta} \parallel \mu'(t)) \parallel_n \mathrm{d}t \\
&\geq k_1 \parallel \int_a^{\beta}  \mu'(t)  \mathrm{d}t \parallel_n
= k_1 \parallel \mu(\beta) - \mu(a) \parallel_n \quad \mathrm{for \: some} \: k_1 > 0.
\end{align*}
(The last inequality follows from~\cite[Theorem 2.1.1]{ham}). Thereby, if $x_0 \neq y_0$ then  $\rho^n(x_0,y_0) > 0$ and hence $\rho (x_0,y_0) > 0$.
Now we prove that the topology induced by $\rho$ coincides with the topology of $M$. By virtue of Lemma~\ref{psm}, we only need to show that 
$\{\rho^n\}_{n \in \nn}$ induces the topology which is consistent with the topology of $M$.
If $x_i \rightarrow x_0$ in $M$ then eventually $x_i \in \mathcal{U}$. 
Define $\lambda_i : [0,1] \rightarrow \mathcal{U}$, an $MC^1$-curve connecting $x_0$ to $x_i$, by $t\varphi(x_i)$. Then, for all $n \in \nn$
\begin{align*}
\rho^n (x_i,x_0) &\leq L^n (\lambda_i) = \int_0^1 \parallel \lambda_i'\parallel^n_{\lambda_i(t)} \mathrm{d}t = 
\int_0^1 \parallel \varphi(x_i) \parallel_{t\varphi(x_i)}^n \mathrm{d}t \\
&\leq K \int_0^1 \parallel \varphi(x_i) \parallel_{x_0}^n \mathrm{d}t  = K \parallel \varphi(x_i)\parallel_n.
\end{align*}
But $\varphi(x_i)\rightarrow 0$ as $x_i \rightarrow x_0$, thereby $\rho^n(x_i,x_0) \rightarrow 0$ for all $n \in \nn$. Conversely, if for all $n \in \nn$, $\rho^n(x_i,x_0) \rightarrow 0$
then eventually we can choose $r$ small enough so that $x_i \in \mathcal{U}$. Then, for all $n \in \nn$ we have 
$ \parallel \varphi(x_i) \parallel_{x_0}^n \leq K\rho^n(x_i,x_0) $ so $ \parallel \varphi(x_i) \parallel_{x_0}^n \rightarrow 0$ in $T_{x_0}M$,
whence $\varphi(x_i) \rightarrow 0$. Therefore,
$x_i \rightarrow x_0$ in $\mathcal{U}$ and hence in $M$.
\end{proof}
The metric $\rho$ is called the Finsler metric for $M$ and it is bounded by 1. 
  
\section{Morse-Sard-Brown Theorem}\label{msb}

In this section we prove the Morse-Sard-Brown theorem for functionals on bounded-Fr\'{e}chet-Finsler manifolds. The proof relies on
the following inverse function theorem.
\begin{theorem}[\cite{glockner}, Proposition 7.1. Inverse Function Theorem for $ MC^k$-maps] \label{invr}
Let $ (E,g) $ be a Fr\'{e}chet space with standard metric $g$. Let $U \subset E$ be open, $x_0 \in U$ and
$ f : U \subset E \rightarrow E $  an $ MC^k$-map, $ k \geq 1 $.
 If $f'(x_0) \in \Aut{(E)}  $, then there exists an open neighborhood
 $ V \subseteq U $ of $ x_0 $ such that $ f(V) $ is open in $ E $ and $ f\vert_V : V \rightarrow f(V) $ 
is an $ MC^k$- diffeomorphism.
\end{theorem}
The following consequence of this theorem is an important technical tool.
\begin{prop}[Local representation]\label{representation}
Let $F_1,F_2$ be Fr\'{e}chet spaces and $U$ an open subset of $F_1 \times F_2$ with $(0,0) \in U$. Let $E_2$ be another Fr\'{e}chet space and $\phi : U \rightarrow F_1 \times E_2$  
 an $MC^{\infty}$-map with $\phi(0,0) = (0,0)$. Assume that the partial derivative $\operatorname{D}_1\phi(0,0) : F_1 \rightarrow F_1$ 
 with respect to the first variable is linear
 isomorphism. Then there exists a local $MC^{\infty}$-diffeomorphism $\psi$ from an open neighborhood $V_1 \times V_2 \subseteq F_1 \times F_2$ of $(0,0)$
 onto an open neighborhood of $(0,0)$ contained in $U$ such that  $\phi \circ \psi (u,v) = (u,\mu (u,v))$, where 
 $\mu : V_1 \times V_2 \rightarrow E_2$ is an $MC^{\infty}$-mapping.
\end{prop}
\begin{proof}
 Let $\phi = \phi_1 \times \phi_2$, where $\phi_1 : U \rightarrow F_1$ and $\phi_2 : U \rightarrow E_2$. By assumption 
we have $ \operatorname {D_1}\phi_1(0,0) = \operatorname{D}_1 \phi(0,0) \vert_{F_1} \in \Iso (F_1, F_1) $.
 Define the map
\begin{equation*}
\begin{array}{cccc}
g : U \subset F_1 \times F_2 \rightarrow F_1 \times E_2 ,\\
g(u_1, u_2) = (\phi_1 (u_1,u_2),u_2)
\end{array}
\end{equation*}
locally at $(0,0)$. Therefore, for all $ u = (u_1,u_2) \in U,\, f_1 \in F_1, \, f_2 \in F_2 $ we have
\begin{equation*}
\operatorname{D}g(u)\cdotp(f_1,f_2)=
\begin{pmatrix}
 \operatorname {D_1}\phi_1(u) &  \operatorname{D_2}\phi_1(u) \\
0 & \mathrm{Id}_{E_2}  
\end{pmatrix}
\begin{pmatrix}
f_1 \\
f_2
\end{pmatrix}
 \end{equation*}
 and hence $\operatorname{D}g(u)$ is a linear isomorphism at $(0,0)$. By the inverse function theorem, there are open sets $ U' $ and $ V = V_1 \times V_2$ and an $ MC^{\infty}$- 
diffeomorphism $ \Psi : V \rightarrow U' $ such that
 $ (0,0) \in U' \subset U, \, g(0,0) \in V \subset F_1 \times E_2$,
 and $ \Psi ^{-1} = g \vert_{ U^{'}} $. Hence if $ (u,v) \in V $, then
 $ (u,v) = (g \circ \Psi)(u,v) = g(\Psi_1(u,v),\Psi_2(u,v))= (\phi_1 \circ \Psi_1(u,v),\Psi_2(u,v)),$
 where $ \Psi = \Psi_1 \times \Psi_2 $. This shows that $ \Psi_2(v,v)=v $ and $ (\phi_1 \circ \Psi)(u,v) = u $.
 Define $ \eta = \phi_2 \circ \Psi $, then 
$$ (\phi \circ \Psi)(u,v) = (\phi_1 \circ \Psi (u,v),\phi_2 \circ \Psi (u,v))=(u, \eta (u,v)). $$ This completes the proof.
\end{proof}
In the sequel, we assume that $M$ is connected and endowed with a Finsler structure $\{\Finsler\}_{n\in \nn}$ and a Finsler metric $\rho$.

\begin{defn}
 Let $l: M \rightarrow \mathbb{R}$ be an $MC^{\infty}$-functional  
 and $\xi : TM \rightarrow M$ a smooth vector field. By saying that $l$ and $\xi$ are
associated we mean $\operatorname{D}l(p) = 0$ if and only if $\xi(p) = 0$. A point $p \in M$ is called a critical point for $l$ if 
$\operatorname {D}l(p) = 0$. The corresponding value $l(p)$ is called a critical value. Values other than critical are called regular values. The set of all critical points of $l$ is denoted by $Crit_l$.
\end{defn}

The following is our version of the compactness condition due to Tromba~\cite{tromba2}.
\begin{con}[\textbf{CV}]
 Let $( m_i)_{i \in \nn}$ be a sequence in $M$. We say that a vector field $\xi :M \rightarrow TM$ satisfies condition \emph{(CV)} if 
 $\parallel \xi(m_i) \parallel^n \rightarrow 0$ for all $n \in \nn$, then
$( m_i)_{i \in \nn}$ has a convergent subsequence.  
\end{con}
It follows immediately from the definition
that the set of zeros of a vector field $ V $ that satisfies the condition (CV) is compact.

A subset $ G $ of a Fr\'{e}chet space $ E $ is called topologically complemented or split in $ E $ if there is
another subspace $H$ of $ E $ such that $E$ is
homeomorphic to the topological direct sum $ G \oplus H $. In this case we call $ H $ a topological complement
of $ G $ in $ F $.

We will need the following facts:
\begin{theorem}[\cite {m}, Theorem 3.14] \label{comp2}
Let $ E $ be a Fr\'{e}chet space. Then
\begin{enumerate}
\item Every finite-dimensional subspace of $ E $ is closed.
\item Every closed subspace $ G \subset E $ with $ \codim(G)=\dim (E/G)< \infty $ is topologically complemented in $ E $.
\item Every finite-dimensional subspace of $ E $ is topologically complemented.
\item Every linear isomorphism between the direct sum of two closed subspaces and $ E $, $ G \oplus H \rightarrow E $, is a homeomorphism.
\end{enumerate}
\end{theorem}
The proof of the Morse-Srad-Brown theorem requires Proposition~\ref{representation} and Theorem~\ref{comp2}. 
Except the arguments which involve these results and the Finslerian nature of manifolds, 
the rest of arguments are similar to that of Banach manifolds case, see~\cite[Theorem 1]{tromba3}. 
\begin{theorem}[Morse-Sard-Brown Theorem]
Assume that $(M,\rho)$ is endowed with a strengthened connection $\mathcal{K}$.  
 Let $\xi$ be a smooth Lipschitz-Fredholm vector field on $M$ with respect to $\mathcal{K}$ which satisfies condition \emph{(CV)}.
 Then, for any $MC^{\infty}$-functional $l$ on $M$ which is associated to $\xi$, the set of its critical values $l(Crit_l)$ 
 is of first category in $\rr$. Therefore, the set of the regular values of $l$ is a residual Baire subset of $\rr$.
\end{theorem}
\begin{proof}
 We can assume $M = \bigcup_{i \in \nn} M_i$, where all the $M_i$'s are closed balls of radius $i$ about some fixed point $ m_0 \in M$ with respect  to the Finsler metric $\rho$.
 Thus to conclude the proof  it suffices to show that the image $l(C_B)$ of 
 the set $C_B$ of the zeros of $\xi$ in some closed set $B$ is compact without interior. 

Let $B$ be a closed set and $C_B$ as before. If $p \in C_B$ then eventually $\xi(p) = 0$.
Since $C_B$ is compact we only need to show that for a  neighborhood $U$ of $p$, $l(C_B \cap \overline{U})$ is compact without interior.
In other words, we can work locally. Therefore, we may assume without loss of generality that $p= 0 \in F$ and $\xi,\, l$ are defined locally on an open neighborhood of $p$.
An endomorphism $\operatorname{D}\xi(p) : F \rightarrow F$ is a Lipschitz-Fredholm operator because $\xi$ is a Lipschitz-Fredholm vector field
(see Remark~\ref{zero}). Thereby,
in the light of Theorem~\ref{comp2}
it has the split image $F_1$ with the topological complement $F_2$ and the split kernel $E_2$ with the topological complement $E_1$. 
Moreover, $\operatorname{D}\xi(p)$ maps $E_1$ isomorphically onto $F_1$ so we can identify $F_1$ with $E_1$. 
Then by Proposition~\ref{representation},
there is an open neighborhood  $ U \subset E_1 \times E_2 $ of $p$
 such that $ \xi (u,v) = (u, \eta(u,v)) $ for all $ (u,v) \in U$, 
where $ \eta : U \rightarrow F_2  $ is an $ MC^{\infty}$-map. Thus, if $\xi(u,v) = 0 = (u, \eta(u,v)) $ then $u = 0$. Therefore,
in this local representation, the zeros of $\xi$  (and hence critical points of $l$) in $\overline{U}$ are in
$\overline{U}_1  \coloneq \overline{U} \cap (\{0 \}\times E_2)$. The restriction of $l$, $l_{\overline{U}_1}:\overline{U}_1 \rightarrow \rr$, 
is again $MC^{\infty}$ and $C_B \cap \overline{U} = C_B \cap \overline{U}_1$  so $l(C_B \cap \overline{U}) = l(C_B \cap \overline{U}_1)$. 

We have for some constant $k \in \nn$, $\dim {\overline{U}_1} = \dim E_2 = k$
because $\xi(p)$ is a Lipschitz-Fredholm operator and $E_2$ is its kernel. Thus, by the classical Sard's theorem $l(C_B \cap \overline{U}_1)$ has measure zero
(note that $MC^k$-differentiability implies the usual $C^k$-differentiability for maps of finite dimensional manifolds).
Therefore, since $C_B \cap \overline{U}_1$ is compact it follows that $l(C_B \cap \overline{U}_1)$ is compact without interior and hence $l(C_B \cap \overline{U})$
is compact without interior. 
\end{proof}
\begin{remk}
From the proceeding proof we have that $\dim {F_2} = m$, where $m \in \nn$ is  constant. Thus, the index of $\xi$ is the following:
 \begin{equation*}
  \Ind \xi = \dim{E_2} - \dim{F_2} = k - m.
 \end{equation*}

\end{remk}

\end{document}